\documentclass[]{article}
\usepackage{amsmath}
\usepackage[scale=0.7]{geometry}
\usepackage{cases}
\usepackage{amssymb,amsfonts,amsthm}
\usepackage{bm}
\usepackage{bigstrut,multirow,rotating}
\usepackage{mathrsfs}
\usepackage{subfigure}
\makeatletter

\newcommand{\Rmnum}[1]{\expandafter\@slowromancap\romannumeral #1@}

\makeatother
\usepackage{authblk}
\usepackage{hyperref}                                                  
\hypersetup{hypertex = true, colorlinks = true, linkcolor = blue, anchorcolor = blue, citecolor =blue}
\usepackage[square,numbers]{natbib}
\newtheorem{theorem}{Theorem}

\newtheorem{remark}{Remark}

\newtheorem{lemma}{Lemma}[section]
\newcommand{\EE}{\mathbb{E}}
\newcommand{\rd}{\mathrm{d}}
\newcommand{\setZ}{\mathbb{Z}}

\newcommand{\setD}{\mathcal{D}}
\newcommand{\setG}{\Gamma}
\newcommand{\vecx}{\boldsymbol{x}}
\newcommand{\veck}{\boldsymbol{k}}
\newcommand{\vecb}{\boldsymbol{b}}
\newcommand{\vecu}{\boldsymbol{u}}

\newcommand{\vecDelta}{\boldsymbol{\Delta}}
\newcommand{\vecdelta}{\boldsymbol{\delta}}
\newcommand{\indicator}[1]{1_{#1}}
{\bgroup
	\addtolength\abovedisplayshortskip{#1}
	\addtolength\abovedisplayskip{#1}
	\addtolength\belowdisplayshortskip{#1}
	\addtolength\belowdisplayskip{#1}}
{\egroup\ignorespacesafterend}
\usepackage{caption}
\usepackage{booktabs}
\usepackage{lipsum}
\usepackage{amsfonts}
\usepackage{graphicx}
\usepackage{epstopdf}
\usepackage{algorithm}
\usepackage{algorithmic}
\usepackage{color}
\usepackage{enumitem}
\usepackage{dsfont}

\title{The inverse of the star discrepancy of a union of randomly shifted Korobov rank-1 lattice point sets depends polynomially on the dimension}

\author[1]{Jiarui Du}
\author[2]{Josef Dick \thanks{Corresponding author: josef.dick@unsw.edu.au}}

\affil[1]{School of Mathematics, South China University of Technology, Guangzhou 510641, Guangdong, People’s Republic of China}
\affil[2]{School of Mathematics and Statistics, The University of New South Wales, Kensington, NSW 2052, Australia}
\begin{document}
	
	\maketitle
	
\begin{abstract}
The inverse of the star discrepancy, $N(\epsilon, s)$, defined as the minimum number of points required to achieve a star discrepancy of at most $\epsilon$ in dimension $s$, is known to depend linearly on $s$. However, explicit constructions achieving this optimal linear dependence remain elusive. Recently, Dick and Pillichshammer (2025) made significant progress by showing that a multiset union of randomly digitally shifted Korobov polynomial lattice point sets almost achieve the optimal dimension dependence with high probability. 
    
In this paper, we investigate the analog of this result in the setting of classical integer arithmetic using Fourier analysis. We analyze point sets constructed as multiset unions of Korobov rank-1 lattice point sets modulo a prime $N$. We provide a comprehensive analysis covering four distinct construction scenarios, combining either random or fixed integer generators with either continuous torus shifts or discrete grid shifts. We prove that in all four cases, the star discrepancy is bounded by a term of order $O(s \log(N_{tot}) / \sqrt{N_{tot}})$ with high probability, where $N_{tot}$ is the total number of points. This implies that the inverse of the star discrepancy for these structured sets depends quadratically on the dimension $s$. While the proofs are probabilistic, our results significantly reduce the search space for optimal point sets from a continuum to a finite set of candidates parameterized by integer generators and random shifts.
\end{abstract}
\noindent\textbf{Key words}:
Star discrepancy, Korobov rank-1 lattice rules, Fourier analysis, quasi-Monte Carlo


\section{Introduction}
    Let $s, N \in \mathbb{N}$. For a finite set $E = \{\boldsymbol{x}_0, \dots, \boldsymbol{x}_{N-1}\}$ in $[0, 1)^s$, the star discrepancy is defined as
\begin{equation} \label{eq:star_discrepancy}
    D_N^*(E) := \sup_{\boldsymbol{b} \in [0, 1]^s} \left| \frac{1}{N} \sum_{n=0}^{N-1} 1_{[\boldsymbol{0}, \boldsymbol{b})}(\boldsymbol{x}_n) - \lambda([\boldsymbol{0}, \boldsymbol{b})) \right|, \nonumber
\end{equation}
where $1_{[\boldsymbol{0}, \boldsymbol{b})}$ denotes the indicator function of the axis-parallel box $[\boldsymbol{0}, \boldsymbol{b})$ and $\lambda$ denotes the Lebesgue measure. This quantity is a central complexity parameter in quasi-Monte Carlo (QMC) integration. 
The classical Koksma-Hlawka (K-H) inequality provides an error bound on the integration error \cite{dick2010,owenqmc}: 
		\begin{align*}
				\left|\int_{(0,1)^d} f({\bm{u}}) \mathrm{d} {\bm{u}}-\frac{1}{N} \sum_{i=1}^N f(\bm{u}_{i})\right|\leq D_N^*\left(\bm{u}_{1}, \ldots, \bm{u}_{N}\right) V_{\mathrm{HK}}(f),
			\end{align*}
		where $V_{\mathrm{HK}}(f)$ denotes the variation of $f$ in the sense of Hardy and Krause. For its definition and properties, we refer to \cite{owen2005hk}.
By the Koksma-Hlawka inequality, a small star discrepancy implies a small worst-case integration error for functions of bounded variation.

A fundamental quantity of interest is the inverse of the star discrepancy,
\begin{equation}
    N(\epsilon, s) := \min \{ N \in \mathbb{N} : \exists E \subset [0, 1)^s, |E|=N \text{ such that } D_N^*(E) \le \epsilon \}, \nonumber
\end{equation}
for $\epsilon \in (0, 1]$ and $s \in \mathbb{N}$.
Theoretical investigations have confirmed that the inverse star discrepancy $N(\varepsilon, s)$ exhibits a linear dependence on the dimension $s$. As established by Heinrich et al.~\cite{heinrich2000}, there exists a universal constant $C$ such that
\begin{equation*}
    N(\varepsilon, s) \le C s \varepsilon^{-2}.
\end{equation*}
This upper bound demonstrates that there exists a point set $P$ with $N$ points belonging to $[0,1)^s$ such that  $D^*_N(P) \le c\sqrt{s/N}$ for some constant $c>0$. The tightness of this result was subsequently verified by Hinrichs~\cite{Hinrichs2004}, who proved a matching lower bound 
$$
N(\varepsilon, s) \ge cs/\epsilon.
$$ for sufficiently small $\epsilon$, thereby confirming the linear barrier in $s$. Elementary arguments yielding this linear-in-$s$ barrier were subsequently given by Steinerberger \cite{steinerberger2023}.

Despite strong existential results, constructing explicit deterministic point sets with optimal star discrepancy remains an open problem. Classical QMC sequences, including Sobol', Halton, rank-1 lattices, and polynomial lattices, achieve discrepancy bounds of order $(O(N^{-1}(\log N)^{s-1}))$, which deteriorate rapidly with increasing dimension $s$, despite being asymptotically better than i.i.d. sampling for fixed $s$. Although weighted spaces and component-by-component constructions provide partial remedies \cite{aistleitner2014, hinrichs2008}, the optimal $\sqrt{s/N}$ rate for the unweighted star discrepancy has not been attained. Moreover, since computing $D_N^*(P)$ is NP-hard \cite{Gnewuch2009}, practical approaches rely on randomized and heuristic methods \cite{clement2022star,clement2023computing,gross2021}, underscoring the need for new structural insights into explicit constructions.



In a recent paper, Dick and Pillichshammer \cite{dick2025inverse} made a step towards constructive solutions by investigating multiset unions of digitally shifted Korobov polynomial lattice point sets. Relying on polynomial arithmetic over finite fields and Walsh analysis, they proved that such structured sets can achieve the optimal linear dimension dependence with high probability. This result narrows the search for optimal points from the entire unit cube to a finite family of polynomial lattices.

In this paper, we extend these results to the setting of classical integer arithmetic. We investigate point sets obtained as multiset unions of randomly shifted Korobov rank-1 lattice point sets modulo a prime $N$. Unlike the polynomial case, our analysis relies on Fourier analysis on $\mathbb{Z}_N$ to control the character sums associated with the lattice points. The concept of multiple rank-1 lattices has emerged as a powerful tool, originally motivated by challenges in numerical integration in Wiener spaces \cite{D14, Goda23} and in sparse trigonometric approximation and high-dimensional Fourier sampling~\cite{gross2021,kammerer2018,kammerer2019,kammerer2021}. 

Our main contributions are as follows. We establish discrepancy bounds for four distinct scenarios: (i) random generators with continuous shifts, (ii) fixed generators with continuous shifts, (iii) random generators with discrete shifts, and (iv) fixed generators with discrete shifts. We show that in all cases, the star discrepancy decays at a rate of $O(s \log(N_{tot}) / \sqrt{N_{tot}})$, implying a quadratic dependence of $N(\epsilon, s)$ on the dimension for these constructions.

\section{Notation and Definitions}
Throughout this paper, let $N$ be a prime number. Define $\mathcal{D}_N = \{0, 1, \dots, N-1\}$ and $\Gamma = \{0,1/N,2/N,\ldots,(N-1)/N\}$ and $\bar{\Gamma} = \Gamma \bigcup \{1\}$.

\textbf{Local discrepancy} For dimension $s \in \mathbb{N}$ and an axis-parallel box $J(\boldsymbol{b}) = [\boldsymbol{0}, \boldsymbol{b}) = \prod_{j=1}^s [0, b_j)$ with $\boldsymbol{b} \in [0, 1]^s$, the local discrepancy of a point set $E$ in $J(\boldsymbol{b})$ is given by
\begin{equation}\label{eq:local}
    \mathrm{disc}(E, J(\boldsymbol{b})) := \frac{1}{|E|} \sum_{\boldsymbol{x} \in E} 1_{J(\boldsymbol{b})}(\boldsymbol{x}) - \lambda(J(\boldsymbol{b})).
\end{equation}

\textbf{Korobov Rank-1 Lattice Rules}
We consider Korobov rank-1 lattice rules modulo a prime $N$. A Korobov rank-1 lattice point set $P_N(z)$ is determined by an integer $z \in \{1, \dots, N-1\}$,
\begin{equation}\label{eq:pnz}
    P_N(z) := \left\{ \frac{n \times a(z) \pmod N} {N}  : n = 0, 1, \dots, N-1 \right\}, \nonumber
\end{equation}
where $a(z) = (1, z, z^2, \dots, z^{s-1})$.

\textbf{Random Shifts}  
To randomize the construction, we employ two types of shifts.
\begin{itemize}
		\item \textbf{Continuous Shift}: Let $\vecDelta \in [0, 1)^s$. The shifted set is defined as:\begin{equation}\label{eq:shift}
    P_N(z) \oplus \boldsymbol{\Delta} := \left\{ \{ \boldsymbol{x} + \boldsymbol{\Delta} \} : \boldsymbol{x} \in P_N(z) \right\}, \nonumber
\end{equation}
where $\{\cdot\}$ denotes the fractional part applied component-wise. If $\boldsymbol{\Delta}$ is uniformly distributed in $[0, 1)^s$, the shifted set becomes a randomization of the original lattice.
		\item \textbf{Discrete Shift}: Let $\vecdelta \in \setG^s$. The shifted set is $P_N(z) \oplus \vecdelta$. Note that since both the Korobov lattice points $P_N(z)$ and the random shift $\boldsymbol{\delta}$ belong to the grid $\Gamma^s$, the shifted point set $P_N(z) \oplus \boldsymbol{\delta}$ remains entirely within $\Gamma^s$.
	\end{itemize}
    
\section{Fourier Analysis Framework}
	
	Our analysis relies on Fourier series expansions of the indicator function $1_{J(\vecb)}$. We present the framework for both the continuous setting (for continuous shifts) and the discrete setting (for discrete shifts).
	
	\subsection{Continuous Setting}
	
	For $\boldsymbol{k} \in \mathbb{Z}^s$, the Fourier coefficient of a function $f \in L^2([0, 1]^s)$ is defined as $\hat{f}(\boldsymbol{k}) = \int_{[0, 1]^s} f(\boldsymbol{x}) e^{-2\pi i \boldsymbol{k} \cdot \boldsymbol{x}} \, \mathrm{d}\boldsymbol{x}$. 
	
	\begin{lemma} \label{lem:cont_fourier_coeffs}
Let $J(\vecb) = [\boldsymbol{0}, \vecb) = \prod_{j=1}^s [0, b_j)$ be an axis-parallel box with $\vecb \in [0,1]^s$. The indicator function $1_{J(\vecb)}(\vecx)$ has the Fourier series expansion (in the $L^2$-sense)
\begin{equation} \label{eq:fourier_series}
    1_{J(\vecb)}(\vecx) = \sum_{\veck \in \setZ^s} c_{\veck}(\vecb) e^{2\pi i \veck \cdot \vecx}, \nonumber 
\end{equation}
where the Fourier coefficients are given by $c_{\veck}(\vecb) = \prod_{j=1}^s c_{k_j}(b_j)$, with
\begin{equation} \label{eq:1d_coeffs}
    c_k(b) = \begin{cases}
        b & \text{if } k = 0, \\
        \frac{1 - e^{-2\pi i k b}}{2\pi i k} & \text{if } k \neq 0. \nonumber
    \end{cases}
\end{equation}
In particular, $c_{\boldsymbol{0}}(\vecb) = \prod_{j=1}^s b_j = \lambda(J(\vecb))$.
\end{lemma}

\begin{proof}
Since the indicator function factorizes as $1_{J(\vecb)}(\vecx) = \prod_{j=1}^s 1_{[0, b_j)}(x_j)$, the Fourier coefficients also factorize as the product of the one-dimensional coefficients. It suffices to calculate the coefficient $c_k(b)$ for the one-dimensional indicator function $1_{[0,b)}(x)$. By definition,
\begin{equation*}
    c_k(b) = \int_{0}^{1} 1_{[0,b)}(x) e^{-2\pi i k x} \, \rd x = \int_{0}^{b} e^{-2\pi i k x} \, \rd x.
\end{equation*}
For $k=0$, we immediately obtain $c_0(b) = \int_0^b 1 \, \rd x = b$. For $k \neq 0$, integration yields
\begin{equation*}
    c_k(b) = \left[ \frac{e^{-2\pi i k x}}{-2\pi i k} \right]_{0}^{b} = \frac{e^{-2\pi i k b} - 1}{-2\pi i k} = \frac{1 - e^{-2\pi i k b}}{2\pi i k}.
\end{equation*}
This completes the proof.
\end{proof}
	
	\begin{lemma} \label{lem:cont_ex_delta}
Let $E = \{\vecx_0, \dots, \vecx_{N-1}\} \subset [0,1)^s$ be an arbitrary finite point set. Let $\vecDelta$ be a random vector uniformly distributed in $[0,1)^s$. Consider the randomly shifted point set defined by $E \oplus \vecDelta := \{ \{\vecx_n + \vecDelta\} : n = 0, \dots, N-1 \}$, where $\{\cdot\}$ denotes the component-wise fractional part. Then, for any axis-parallel box $J(\vecb)$ with $\vecb \in [0,1]^s$, we have
\begin{equation*}
    \mathbb{E}_{\vecDelta} [\mathrm{disc}(E \oplus \vecDelta, J(\vecb))] = 0.
\end{equation*}
\end{lemma}

\begin{proof}
By the definition of the local discrepancy function, the expected discrepancy is given by
\begin{equation*}
\begin{aligned}
    \mathbb{E}_{\vecDelta} [\mathrm{disc}(E \oplus \vecDelta, J(\vecb))] &= \frac{1}{N} \sum_{n=0}^{N-1} \mathbb{E}_{\vecDelta} \left[ 1_{J(\vecb)}(\{\vecx_n + \vecDelta\}) \right] - \lambda(J(\vecb))\\
    & = \frac{1}{N} \sum_{n=0}^{N-1} \mathbb{E}_{\vecDelta} \left[ 1_{J(\vecb)}(\vecDelta) \right]-\lambda(J(\vecb))\\
    &=\lambda(J(\vecb))-\lambda(J(\vecb)) = 0.
\end{aligned}
\end{equation*}
This completes the proof.
\end{proof}
	
	\begin{lemma} \label{lem:con_square}
For any $\vecb \in [0,1]^s$, the Fourier coefficients of the indicator function $1_{J(\vecb)}$ satisfy
\begin{equation} \label{eq:sum_sq_coeffs}
    \sum_{\veck \in \setZ^s \setminus \{\boldsymbol{0}\}} |c_{\veck}(\vecb)|^2 = \lambda(J(\vecb)) (1 - \lambda(J(\vecb))). \nonumber
\end{equation}
\end{lemma}

\begin{proof}
    We apply Parseval's identity for $1_{J(\vecb)}(\vecx)$, which states that the $L^2$-norm of a function is equal to the sum of the squared moduli of its Fourier coefficients \cite{stein2011}:
\begin{equation*}
    \int_{[0,1]^s} |1_{J(\vecb)}(\vecx)|^2 \, \rd \vecx = \sum_{\veck \in \setZ^s} |c_{\veck}(\vecb)|^2,
\end{equation*}
which means that
\begin{equation*}
    \lambda(J(\vecb)) = |c_{\boldsymbol{0}}(\vecb)|^2 + \sum_{\veck \in \setZ^s \setminus \{\boldsymbol{0}\}} |c_{\veck}(\vecb)|^2 = (\lambda(J(\vecb)))^2 + \sum_{\veck \in \setZ^s \setminus \{\boldsymbol{0}\}} |c_{\veck}(\vecb)|^2.
\end{equation*}
Rearranging the terms completes the proof.
\end{proof}

This identity allows us to express the variance of the discrepancy in terms of the decay of Fourier coefficients and the character sums of the lattice points.
	
	\subsection{Discrete Setting}
	
	For functions defined on the grid $\setG^s$, we use the discrete Fourier transform. Let $L^2(\Gamma^s)$ denote the Hilbert space of complex-valued functions on $\Gamma^s$ equipped with the inner product:
\begin{equation} \label{eq:inner_product}
    \langle f, g \rangle_{\Gamma} := \frac{1}{N^s} \sum_{\boldsymbol{x} \in \Gamma^s} f(\boldsymbol{x}) \overline{g(\boldsymbol{x})} = \mathbb{E}_{\boldsymbol{x} \in \Gamma^s} [f(\boldsymbol{x}) \overline{g(\boldsymbol{x})}]. \nonumber
\end{equation}
The discrete Fourier coefficients of a function $f \in L^2(\Gamma^s)$ are defined for frequencies $\boldsymbol{k} \in \mathcal{D}_N^s = \{0, 1, \dots, N-1\}^s$ as:
\begin{equation} \label{eq:dft_def}
    C_{\boldsymbol{k}}(f) := \langle f, e_{\boldsymbol{k}} \rangle_{\Gamma} = \frac{1}{N^s} \sum_{\boldsymbol{x} \in \Gamma^s} f(\boldsymbol{x}) e^{-2\pi i \boldsymbol{k} \cdot \boldsymbol{x}},
\end{equation}
where $e_{\boldsymbol{k}}(\boldsymbol{x}) = e^{2\pi i \boldsymbol{k} \cdot \boldsymbol{x}}$. The function can be reconstructed via the inverse transform:
\begin{equation*}
    f(\boldsymbol{x}) = \sum_{\boldsymbol{k} \in \mathcal{D}_N^s} C_{\boldsymbol{k}}(f) e^{2\pi i \boldsymbol{k} \cdot \boldsymbol{x}}.
\end{equation*}
Crucially, the discrete Parseval identity holds \cite{stein2011}:
\begin{equation} \label{eq:discrete_parseval}
    \sum_{\boldsymbol{k} \in \mathcal{D}_N^s} |C_{\boldsymbol{k}}(f)|^2 = \mathbb{E}_{\boldsymbol{x} \in \Gamma^s} [|f(\boldsymbol{x})|^2]. 
\end{equation}

\begin{lemma}\label{lem:disc_fourier_coeff}
    Let $J(\boldsymbol{b}) = [\bm{0}, \vecb)$ with $\vecb \in \Gamma^s$. The discrete indicator function $1_{J(\boldsymbol{b})}(\boldsymbol{x})$ with $\boldsymbol{x} \in \Gamma^s$ has the discrete Fourier series expansion (in the $L^2$-sense):
    \begin{equation} \label{eq:indicator_dft}
    1_{J(\boldsymbol{b})}(\boldsymbol{x}) = \sum_{\boldsymbol{k} \in \mathcal{D}_N^s} C_{\boldsymbol{k}}(\vecb) e^{2\pi i \boldsymbol{k} \cdot \boldsymbol{x}}, \nonumber
    \end{equation}
    where $\mathcal{D}_N^s = \{0, 1, \dots, N-1\}^s$ and the Fourier coefficients are given by $C_{\veck}(\vecb) = \prod_{j=1}^sC_{k_j}(b_j)$ with
    \begin{equation}
        C_{k}(b) = 
        \begin{cases} 
            b & \text{if } k = 0, \\
            \displaystyle \frac{1}{N} \frac{1 - e^{-2\pi i k b}}{1 - e^{-2\pi i k / N}} & \text{if } k \neq 0.
        \end{cases}
    \end{equation}
\end{lemma}

\begin{proof}
    Since $1_{J(\boldsymbol{b})}(\boldsymbol{x}) = \prod_{j=1}^s 1_{[0, b_j)}(x_j)$, we focus on the one dimensional case. Let $b = m/N$ for some integer $0 \le m \le N$. By definition \eqref{eq:dft_def}, we have
\begin{equation*}
    C_k(b) = \frac{1}{N}\sum_{x\in \Gamma} 1_{[0,b)}(x) e^{-2\pi i k x} = \frac{1}{N} \sum_{n=0}^{m-1} \left( e^{-2\pi i k / N} \right)^n.
\end{equation*}
For $k=0$, we obtain $C_0(b) = \frac{1}{N} \sum_{n=0}^{m-1} 1 = m/N = b$. For $k \neq 0$, the sum is a geometric series with ratio $q = e^{-2\pi i k / N} \neq 1$, we have
\begin{equation*}
    C_k(b) = \frac{1}{N} \frac{1 - (e^{-2\pi i k / N})^m}{1 - e^{-2\pi i k / N}} = \frac{1}{N} \frac{1 - e^{-2\pi i k b}}{1 - e^{-2\pi i k / N}}.
\end{equation*}
This completes the proof.
\end{proof}

\begin{lemma} \label{lem:disc_ex_delta}
Let $E = \{\vecx_0, \dots, \vecx_{N-1}\} \subset [0,1)^s$ be an arbitrary finite point set. Let $\bm{\delta}$ be a random vector uniformly distributed in $\Gamma^s$. Then, for any axis-parallel box $J(\vecb)$ with $\vecb \in \bar\Gamma^s$, we have
\begin{equation*}
    \mathbb{E}_{\bm{\delta}} [\mathrm{disc}(E \oplus \bm{\delta}, J(\vecb))] = 0.
\end{equation*}
\end{lemma}

\begin{proof}
By the definition of the local discrepancy function \eqref{eq:local}, the expected discrepancy is given by
$$
\mathbb{E}_{\bm{\delta}} [\mathrm{disc}(E \oplus \bm{\delta}, J(\vecb))] = \frac{1}{N} \sum_{n=0}^{N-1} \mathbb{E}_{\bm{\delta}} \left[ 1_{J(\vecb)}(\{\vecx_n + \bm{\delta}\}) \right] - \lambda(J(\vecb)).
$$
Since $\bm{\delta} \in \Gamma^s$ and 
$\vecb \in \bar\Gamma^s$, $\mathbb{E}_{\bm{\delta}} \left[ 1_{J(\vecb)}(\{\vecx_n + \bm{\delta}\}) \right] = \mathbb{E}_{\bm{\delta}} \left[ 1_{J(\vecb)}(\bm{\delta}) \right] = \lambda(J(\vecb))$. This completes the proof.
\end{proof}

\begin{lemma}\label{lem:disc_square}
    For any $\vecb \in \bar\Gamma^s$, the discrete Fourier coefficients of the indicator function $1_{J(\vecb)}$ satisfy
    \begin{equation*} 
        \sum_{\boldsymbol{k} \in \mathcal{D}_N^s \setminus \{\boldsymbol{0}\}} |C_{\boldsymbol{k}}(\boldsymbol{b})|^2 = \lambda(J(\boldsymbol{b})) \left( 1 - \lambda(J(\boldsymbol{b})) \right).
    \end{equation*}
\end{lemma}

\begin{proof}
    By \eqref{eq:discrete_parseval}, we apply the discrete Parseval identity to the indicator function $ 1_{J(\boldsymbol{b})}(\boldsymbol{x})$:
    \begin{equation}
        \sum_{\boldsymbol{k} \in \mathcal{D}_N^s} |C_{\boldsymbol{k}}(\vecb)|^2 = \mathbb{E}_{\boldsymbol{x} \in \Gamma^s} [|1_{J(\boldsymbol{b})}(\boldsymbol{x})|^2] = \mathbb{E}_{\boldsymbol{x} \in \Gamma^s} [1_{J(\boldsymbol{b})}(\boldsymbol{x})] = \lambda(J(\boldsymbol{b})). \nonumber
    \end{equation}
    Hence, 
    \begin{equation}
        \sum_{\boldsymbol{k} \in \mathcal{D}_N^s \setminus \{\boldsymbol{0}\}} |C_{\boldsymbol{k}}(\vecb)|^2 = \sum_{\boldsymbol{k} \in \mathcal{D}_N^s} |C_{\boldsymbol{k}}(\vecb)|^2 - |C_{\boldsymbol{0}}(\vecb)|^2 = \lambda(J(\boldsymbol{b})) \left( 1 - \lambda(J(\boldsymbol{b})) \right). \nonumber
    \end{equation}
    This completes the proof.
\end{proof}

\section{Arithmetic Properties of Korobov Lattice Point Sets}
	
The performance of the Korobov lattice point sets depends on the character sums $S_N(z, \veck)$, which is defined by $$ S_N(z, \veck) := \frac{1}{N} \sum_{n=0}^{N-1} e^{2\pi i \frac{n}{N} (\veck \cdot a(z))} = \indicator{\veck \cdot a(z) \equiv 0 \pmod N}. $$

Let $K_{1}$ denote $( \setZ^s \setminus (N\setZ)^s ) \setminus \{\overline{\bm{0}} \}$ which we, for simplicity, identify with the set $\setD_N^s \setminus \{\boldsymbol{0}\}$. For $\veck \in K_{1}$, the polynomial $P_{\veck}(z) = \sum_{j=1}^s k_j z^{j-1}$ is not the zero polynomial modulo $N$.
	
	\begin{lemma}\label{lem:ex_sn}
		Let $z$ be chosen uniformly from $\{1, \dots, M\}$. Then for any $\veck \in K_{1}$,
		$$ 
        \EE_z [|S_N(z, \veck)|^2] \le \frac{s-1}{M}.
        $$
	\end{lemma}
    
	\begin{proof}
		Since $\veck \in K_{1}$, the polynomial $\sum_{j=1}^s k_j z^{j-1} \equiv 0 \pmod N$ has at most $s-1$ solutions in the field $\setZ_N$ by the Fundamental Theorem of Algebra. Thus,
\begin{equation*}
    \EE_z [|S_N(z, \veck)|^2] = \frac{1}{M} \sum_{z=1}^M \indicator{\veck \cdot a(z) \equiv 0 \pmod N}\le \frac{s-1}{M}.
\end{equation*}
This completes the proof.
	\end{proof}
	
\begin{lemma}\label{lem:sum_sn}
Consider the set of all generators $z \in \{1, \dots, M\}$. Then for any $\veck \in K_{1}$,
		$$ 
        \sum_{r=1}^M |S_N(r, \veck)|^2 = \sum_{r=1}^M \indicator{\veck \cdot a(r) \equiv 0 \pmod N} \le s-1.
        $$
	\end{lemma}
    
	\begin{proof}
		This follows immediately from the same argument as Lemma \ref{lem:ex_sn}.
	\end{proof}

\section{Main results}
	We present four theorems covering the combinations of generator selection (Random vs. Fixed) and shift type (Continuous vs. Discrete). 

    \begin{table}[h]
		\centering
		\begin{tabular}{@{}lll@{}}
			\toprule
			Generator $z$ & Shift Type & Result \\ \midrule
			Random $z_1, \dots, z_M$ & Continuous $\vecDelta \in [0,1)^s$ & Theorem \ref{thm:main_result1} \\
            Fixed $1, \dots, M$ & Continuous $\vecDelta \in [0,1)^s$ & Theorem \ref{thm:main_result2} \\
            Random $z_1, \dots, z_M$ & Discrete $\vecdelta \in \setG^s$ & Theorem \ref{thm:main_result3}\\
			Fixed $1, \dots, M$ & Discrete $\vecdelta \in \setG^s$ & Theorem \ref{thm:main_result4}
			 \\ \bottomrule
		\end{tabular}
		\caption{Summary of main results.}
	\end{table}
    
\begin{theorem}[Random $z$, Continuous Shift] \label{thm:main_result1}
Let $N$ be a prime, $M=N-1$, and the total number of points be $N_{tot} = M \cdot N$ (counting multiplicity). Let $z_1, \dots, z_M$ be i.i.d. uniform in $\{1, \dots, M \}$ and $\vecDelta_1, \dots, \vecDelta_M$ be i.i.d. uniform in $[0,1)^s$. Define the multiset union $P := \bigcup_{r=1}^{M} (P_N(z_r) \oplus \vecDelta_r)$.
Then, for every $\delta \in (0,1)$, with probability at least $\delta$, the star discrepancy of $P$ satisfies
\begin{equation} \label{eq:discrepancy_boundp}
    D_{N_{\text{tot}}}^*(P) \le 1.83 \times \frac{s\log(N+1) + \log 2 - \log(1-\delta) }{M}. \nonumber
\end{equation}
\end{theorem}

\begin{theorem}[Fixed $z$, Continuous Shift]\label{thm:main_result2}
Let $N$ be a prime, $M=N-1$, and the total number of points be $N_{tot} = M \cdot N$ (counting multiplicity). Let the generators be fixed as $\{1, \dots, M\}$. Let $\vecDelta_1, \dots, \vecDelta_M$ be i.i.d. uniform in $[0,1)^s$. Define $P_{f} := \bigcup_{r=1}^M (P_N(r) \oplus \vecDelta_r)$.
		Then, for every $\delta \in (0,1)$, with probability at least $1-\delta$, the star discrepancy of $P_{f}$ satisfies
		$$ D_{N_{tot}}^*(P_{f}) \le 1.83 \times \frac{s\log(N+1) + \log 2 - \log(1-\delta) }{M}. $$
	\end{theorem}

    \begin{theorem}[Random $z$, Discrete Shift]\label{thm:main_result3}
Let $N$ be a prime, $M=N-1$, and the total number of points be $N_{tot} = M \cdot N$ (counting multiplicity). Let $z_1, \dots, z_M$ be i.i.d. uniform in $\{1, \dots, M\}$ and $\vecdelta_1, \dots, \vecdelta_M$ be i.i.d. uniform in $\setG^s$. Define $Q := \bigcup_{r=1}^M (P_N(z_r) \oplus \vecdelta_r)$. Then, for every $\delta \in (0,1)$, with probability at least $1-\delta$, the star discrepancy of $Q$ satisfies
		$$ D_{N_{tot}}^*(Q) \le 1.73 \times \frac{s\log(N+1) + \log 2 - \log(1-\delta) }{M}. $$
	\end{theorem}

\begin{theorem}[Fixed $z$, Discrete Shift] \label{thm:main_result4}
Let $N$ be a prime, $M=N-1$, and the total number of points be $N_{tot} = M \cdot N$ (counting multiplicity). Let $\vecDelta_1, \dots, \vecDelta_M$ be i.i.d. uniform in $\Gamma^s$. Define the multiset union $Q_f := \bigcup_{r=1}^{M} (P_N(r) \oplus \vecDelta_r)$. Then, for every $\delta \in (0,1)$, with probability at least $\delta$, the star discrepancy of $Q_f$ satisfies
\begin{equation} \label{eq:discrepancy_boundq}
    D_{N_{\text{tot}}}^*(Q_f) \le 1.73 \times \frac{s\log(N+1) + \log 2 - \log(1-\delta) }{M}. \nonumber
\end{equation}
\end{theorem}

	\begin{remark}
		In all four cases, the bound scales roughly as $O(s \log N_{tot} / \sqrt{N_{tot}})$. This confirms that the inverse of the star discrepancy depends quadratically on $s$, matching the polynomial dependence observed in the polynomial lattice setting \cite{dick2025inverse}. Moreover, the choice of $M$ in Theorems \ref{thm:main_result1}-\ref{thm:main_result4} is made to utilize the full set of generators in the field $\mathbb{Z}_N$ (for the fixed case). However, our proofs remain valid for any $M$ proportional to $N$ (e.g., $M = N/2$). The crucial requirement is $M \asymp N$ (i.e., $M$ is proportional to $N$) to ensure the optimal decay rate of the resulting discrepancy bound.
	\end{remark}

    \section{Proofs}
\subsection{Proof of Theorem \ref{thm:main_result1}}
For the point set $P=\bigcup_{r=1}^{M}\left(P_N(z_r) \oplus \vecDelta_r\right)$, where $z_1, \dots, z_M \in \{1, \dots, M\}$ and $\vecDelta_1, \dots, \vecDelta_M \in [0, 1)^s$ are chosen i.i.d. uniformly distributed, respectively. Let $J(\vecb) = [\boldsymbol{0}, \vecb)$ be an axis-parallel box. We define the random variable $Y_r(\vecb)$ as the local discrepancy of the $r$-th shifted lattice:
\begin{equation*}
    Y_r(\vecb) := \mathrm{disc}(P_N(z_r) \oplus \vecDelta_r, J(\vecb)). 
\end{equation*}
To bound the star discrepancy $D_{N_{tot}}^*(P) = \sup\limits_{\vecb \in [0,1]^s} |\mathrm{disc}(P, J(\vecb))|$, by \cite[Proposition 3.17]{dick2010}, we know that 
\begin{equation}\label{eq:discre}
    D_{N_{tot}}^*(P) \le 
    \max_{\vecb \in \bar\Gamma^s} |\mathrm{disc}(P, J(\vecb))| + \frac{s}{N} =  \max_{\vecb \in \bar\Gamma^s}\left|\frac{1}{M} \sum_{r=1}^{M} Y_r(\vecb)\right|+\frac{s}{N}.
\end{equation}
In the following we employ Bennett's inequality to demonstrate that one can find suitable $\{\left(z_1, \vecDelta_1\right),\\ \ldots,\left(z_M, \vecDelta_M\right)\}$ with high probability such that the star discrepancy of $P$ is small. To be able to apply Bennett's inequality, we need a bound on the variance of $Y_r(\vecb)$.

\begin{lemma} \label{lem:cont_var}
Let $z$ be chosen uniformly at random from $\{1, \dots, M\}$ and let $\vecDelta$ be chosen uniformly at random from $[0, 1)^s$. For any axis-parallel box $J(\vecb)$ with $\vecb \in \bar\Gamma^s$,
\begin{equation} \label{eq:variance_bound}
    \mathbb{E}_{z, \vecDelta} [\mathrm{disc}^2(P_N(z) \oplus \vecDelta, J(\vecb))] \le \frac{s}{M} + \frac{s}{3N^2}. \nonumber
\end{equation}
\end{lemma}

\begin{proof}
We first fix the generator $z$ and analyze the discrepancy as a function of the random shift $\vecDelta$, denoting it by $F_{z, \bm{b}}(\bm{\Delta})  := \mathrm{disc}(P_N(z) \oplus \vecDelta, J(\vecb))$.
By the definition of the local discrepancy, it is clear that if $\bm{b} = \bm{1}$, then 
$$
F_{z, \bm{b}}(\bm{\Delta}) = \frac{1}{N} \sum_{n=1}^N \mathbf{1}_{J(\bm{1})}(\{ \mathbf{x}_{n,\bm{u}} + \bm{\Delta} \}) - \text{Vol}(J(\bm{1}))  = 0.
$$
Now we consider $\emptyset \neq \bm{u} \subseteq \{1,2,\ldots,s\}$ and $\bm{b} = (\bm{b}_{\bm{u}}, \mathbf{1}_{\bar{\bm{u}}})$, which means that if $j\in \bm{u}$, $b_j \neq 1$  otherwise $b_j = 1$. Then, we have
\begin{align*}
    F_{z, (\bm{b}_{\bm{u}}, \mathbf{1}_{\bar{\bm{u}}})}(\bm{\Delta}) &= \frac{1}{N} \sum_{n=1}^N \mathbf{1}_{J(\bm{b}_{\bm{u}}, \mathbf{1}_{\bar{\bm{u}}})}(\{ \mathbf{x}_{n,\bm{u}} + \bm{\Delta} \}) - \text{Vol}(J(\bm{b}_{\bm{u}}, \mathbf{1}_{\bar{\bm{u}}})) \\
    &= \frac{1}{N} \sum_{n=1}^N \mathbf{1}_{J(\bm{b}_{\bm{u}})}(\{ \mathbf{x}_{n,\bm{u}} + \bm{\Delta}_{\bm{u}} \}) - \text{Vol}(J(\bm{b}_{\bm{u}})) \\
    &= \frac{1}{N} \sum_{n=1}^N \left( \sum_{\bm{k}_{\bm{u}} \neq \mathbf{0}} c_{\bm{k}_{\bm{u}}}(\bm{b}_{\bm{u}}) e^{2\pi i \, \bm{k}_{\bm{u}} \cdot (\mathbf{x}_{n,\bm{u}} + \bm{\Delta}_{\bm{u}})} \right)\\
    &= \sum_{\bm{k}_{\bm{u}} \neq \boldsymbol{0}} \left[c_{\bm{k}_{\bm{u}}}(\bm{b}_{\bm{u}}) S_N(z, \bm{k}_{\bm{u}})\right]e^{2\pi i \bm{k}_{\bm{u}} \cdot \boldsymbol{\Delta}_{\bm{u}}},
\end{align*}
where 
$
S_N(z, \bm{k}_{\bm{u}}) := \frac{1}{N} \sum_{n=0}^{N-1} e^{2\pi i \bm{k}_{\bm{u}} \cdot \boldsymbol{x}_{n,\bm{u}}}$.
Since $F_{z, (\bm{b}_{\bm{u}}, \mathbf{1}_{\bar{u}})}(\bm{\Delta})$ is a finite sum of bounded indicator functions, it is bounded and thus belongs to the Hilbert space $L^2([0,1]^s)$. Consequently, by continuous Parseval's identity, it follows that
\begin{equation*}
    \mathbb{E}_{\vecDelta} [|F_{z, (\bm{b}_{\bm{u}}, \mathbf{1}_{\bar{\bm{u}}})}(\bm{\Delta})|^2] = \int_{[0,1]^s} |F_{z, (\bm{b}_{\bm{u}}, \mathbf{1}_{\bar{u}})}(\bm{\Delta})|^2 \, \rd \vecDelta = \sum_{\bm{k}_{\bm{u}} \in \setZ^{|\bm{u}|} \setminus \{\boldsymbol{0}_{\bm{u}}\}} |c_{\veck_{\bm{u}}}(\vecb_{\bm{u}}) S_N(z, \veck_{\bm{u}})|^2.
\end{equation*}
Taking the expectation over $z$, and noting that $\mathbb{E}_g$ is a finite sum which commutes with the convergent infinite series over $\bm{k}_{\bm{u}}$, we obtain:
\begin{equation*}
    \mathbb{E}_{z, \vecDelta}[|F_{z, (\bm{b}_{\bm{u}}, \mathbf{1}_{\bar{u}})}(\bm{\Delta})|^2] = \sum_{\veck_{\bm{u}} \in \setZ^{|\bm{u}|} \setminus \{\boldsymbol{0}_{\bm{u}}\}} |c_{\veck_{\bm{u}}}(\vecb_{\bm{u}})|^2 \, \mathbb{E}_z [|S_N(z, \veck_{\bm{u}})|^2].
\end{equation*}
Let $L_{1} = \{\veck_{\bm{u}} \in \setZ^{|\bm{u}|} \setminus \{\boldsymbol{0}_{\bm{u}}\} : \veck_{\bm{u}} \not\equiv \boldsymbol{0} \pmod N\}$. For such $\veck_{\bm{u}}$, by Lemma \ref{lem:ex_sn}, we have
\begin{equation*}
\begin{aligned}
    \sum_{\veck_{\bm{u}} \in L_1} |c_{\veck_{\bm{u}}}(\vecb_{\bm{u}})|^2 \, \mathbb{E}_z [|S_N(z, \veck_{\bm{u}})|^2] &\le \frac{s-1}{M} \sum_{\veck_{\bm{u}} \in L_{1}} |c_{\veck_{\bm{u}}}(\vecb_{\bm{u}})|^2 \\
    &\le \frac{s}{M} \lambda(J(\vecb_{\bm{u}}))(1 - \lambda(J(\vecb_{\bm{u}}))) \le \frac{s}{M}. 
\end{aligned}
   \end{equation*}
where the last inequality follows by Lemma \ref{lem:con_square}.

Let $L_{2} = (N\setZ)^{|\bm{u}|} \setminus \{\boldsymbol{0}_{\bm{u}}\}$. We know that $S_N(z, \veck) = \indicator{\veck \cdot (1,z,z^2,\ldots,z^{s-1}) \equiv 0 \pmod N}$. For these frequencies, $k_j \equiv 0 \pmod N$ for all $j$, so the congruence $\veck_{\bm{u}} \cdot \boldsymbol{a}_{\bm{u}}(z) \equiv 0 \pmod N$ holds for all $z$. Thus $\mathbb{E}_z [|S_N(z, \veck_{\bm{u}})|^2] = 1$. The remainder term is
\begin{equation*}
\begin{aligned}
    \mathcal{R}_N(\vecb) &:= \sum_{\veck \in L_{2}} |c_{\veck}(\vecb)|^2 = \sum_{\boldsymbol{h} \in \setZ^{|\bm{u}|} \setminus \{\boldsymbol{0}_{\bm{u}}\}} |c_{N\boldsymbol{h}}(\vecb)|^2 \\
    &= \sum_{\boldsymbol{h} \in \setZ^{|\bm{u}|} \setminus \{\boldsymbol{0}_{\bm{u}}\}} \prod_{j\in \bm{u}} |c_{N h_j}(b_j)|^2 = \sum_{\boldsymbol{h} \in \setZ^{|\bm{u}|}} \prod_{j\in \bm{u}} |c_{N h_j}(b_j)|^2-|c_{\boldsymbol{0}_{\bm{u}}}(\vecb_{\bm{u}})|^2 \\
    &= \prod_{j\in \bm{u}} \left( \sum_{h_j \in \setZ} |c_{N h_j}(b_j)|^2 \right) - \prod_{j\in \bm{u}} |c_0(b_j)|^2.
\end{aligned}
\end{equation*}
where we have use the fact that the full set $\setZ^{|\bm{u}|}$ is a Cartesian product and the series converges absolutely. By Lemma \ref{lem:cont_fourier_coeffs}, for any non-zero integer $k$,
$$
|c_k(b)| = \left| \frac{1 - e^{-2\pi i k b}}{2\pi i k} \right| \le \frac{2}{2\pi |k|} = \frac{1}{\pi |k|}.
$$
Thus, for $h_j \neq 0$,
\begin{equation*}
    \sum_{h_j \in \setZ/\{0\}}|c_{N h_j}(b_j)|^2 \le \sum_{h_j \in \setZ/\{0\}}\frac{1}{\pi^2 N^2 h_j^2} = \frac{2}{\pi^2 N^2} \sum_{h=1}^{\infty} \frac{1}{h^2} = \frac{2}{\pi^2 N^2} \frac{\pi^2}{6} = \frac{1}{3N^2}.
\end{equation*}
For $h_j = 0$, we have $|c_0(b_j)|^2 = b_j^2 \le 1$. Thus,
\begin{align*}
    \mathcal{R}_N(\vecb) &= \left( \prod_{j\in \bm{u}} \sum_{h_j \in \setZ} |c_{N h_j}(b_j)|^2 \right) - \prod_{j\in \bm{u}} |c_0(b_j)|^2 \\
    &= \prod_{j\in \bm{u}} \left( |c_0(b_j)|^2 + \sum_{h_j \neq 0} |c_{N h_j}(b_j)|^2 \right) - \prod_{j\in \bm{u}} b_j^2 \\
    &\le \prod_{j\in \bm{u}} \left( b_j^2 + \frac{1}{3N^2} \right) - \prod_{j\in \bm{u}} b_j^2 .
\end{align*}
Define $f(y) := \prod_{j \in \bm{u}} (b_j^2 + y)$, then, it follows that
$$
\prod_{j\in \bm{u}} \left( b_j^2 + \frac{1}{3N^2} \right) - \prod_{j\in \bm{u}} b_j^2  =  f\left(\frac{1}{3N^2}\right) - f(0) =  \frac{1}{3N^2}f'(\xi).
$$
where $0 \le \xi \le 1/(3N^2)$ by the Mean Value Theorem. Since for $\bm{b}\in \Gamma$ and $j\in \bm{u}$, $b_j \neq 1$, we have $b_j^2+\xi \le (1-1/N)^2+1/(3N^2) = 1-2/N+4/(3N^2) < 1$. Thus,
\begin{equation*}
    f'(z) = \sum_{i \in \bm{u}} \prod_{\substack{j \in \bm{u} \\ j \neq i}} (b_j^2 + \xi) < |\bm{u}| \le s.
\end{equation*}
And, 
$$
\mathcal{R}_N(\vecb) \le \prod_{j\in \bm{u}} \left( b_j^2 + \frac{1}{3N^2} \right) - \prod_{j\in \bm{u}} b_j^2 =  \frac{1}{3N^2}f'(\xi) < \frac{s}{3N^2}.
$$
Combining the estimates for the main term and the remainder term concludes the proof.
\end{proof}


\begin{proof}[Proof of Theorem \ref{thm:main_result1}]
From Lemma \ref{lem:cont_ex_delta}, we have $\mathbb{E}_{\vecDelta_r}[Y_r(\vecb)] = 0$. By Lemma \ref{lem:cont_var}, 
\begin{equation*}
     \text{Var}[Y_r(\vecb)] =  \mathbb{E}_{z_r, \vecDelta_r} [Y_r^2(\vecb)] \le \frac{s}{M} + \frac{s}{3N^2}.
\end{equation*}
For fixed $\vecb$, the random variables $Y_r(\vecb), r=1,2,\ldots,M$, are independent identically distributed. Hence, let $S_M(\vecb) = \sum_{r=1}^M Y_r(\vecb)$,
$$
\sigma^2 := \text{Var}[S_M(\vecb)] = \sum_{r=1}^M\text{Var}[Y_r(\vecb)] \le s \left(1 + \frac{M}{3N^2}  \right) < s (1 + 1/(3N)).
$$
In summary, we have the following properties 
\begin{equation}\label{eq:Yr}
    |Y_r(\vecb)| \le 1,\quad \mathbb{E}_{\vecDelta_r}[Y_r(\vecb)] = 0, \quad \text{Var}\left[\sum_{r=1}^M Y_r(\vecb)\right] < s (1+ 1/(3N) ).
\end{equation}
Thus, we apply Bernstein's inequality \cite{bennett1962} to $S_M(\vecb)$. For any $t > 0$,
\begin{equation*}
    \mathbb{P}[|S_M(\vecb)| > t] \le 2 \exp\left( - \frac{t^2}{2 \sigma^2 + 2t/3} \right) \le 2 \exp\left( - \frac{t^2}{2s(1+1/(3N)) + 2t/3} \right).
\end{equation*}
The grid size is $|\bar{\Gamma}^s| = (N+1)^s$. We apply the union bound over all $\vecb \in \bar{\Gamma}^s$:
\begin{equation*}
    \mathbb{P}\left[ \exists \vecb \in \bar{\Gamma}^s : |S_M(\vecb)| > t \right] \le 2(N+1)^s \exp\left( - \frac{t^2}{2s(1+1/(3N)) + 2t/3} \right).
\end{equation*}
We set the right-hand side equal to $1-\delta$. Taking logarithms, we need $t$ to satisfy:
\begin{equation*}
    \frac{t^2}{2s(1+1/(3N)) + 2t/3} \ge \log \left(2(N+1)^s\right) - \log(1-\delta) =: L.
\end{equation*}
Solving the quadratic equation $t^2 - 2Lt/3 - 2s(1+1/(3N))L = 0$ for positive $t$, we obtain
\begin{equation*}
    t_0 = \frac{L}{3}\left(1 + \sqrt{1 + \frac{18s(1 + 1/(3N))}{L}}\right).
\end{equation*}
Then we have 
$$
\mathbb{P} \left[ \forall \boldsymbol{b} \in \bar\Gamma^s : \frac{1}M |S_M(\vecb)| \le \frac{t_0}{M} \right] \ge \delta > 0.
$$
Thus with probability at least $\delta$, for all intervals $J(\boldsymbol{b}), \boldsymbol{b} \in \Gamma^s$, we have 
\begin{equation}\label{eq:MYr}
    \begin{aligned}
\frac{1}M |S_M(\vecb)| 
&\le \frac{\log(2(N+1)^s) - \log(1-\delta)}{3 M} \left( 1 + \sqrt{1 + \frac{18s(1 + 1/(3N))}{\log(2(N+1)^s) - \log(1-\delta)}} \right) \\
&\le \frac{s\log(N+1) + \log 2 - \log(1-\delta)}{M} \underbrace{\frac{1}{3} \left( 1 + \sqrt{1 + \frac{18(1 + 1/6)}{\log 3}} \right)}_{=1.8283\ldots}. 
\end{aligned}
\end{equation}
By \eqref{eq:discre} and \eqref{eq:MYr}, we obtain that for any $\delta \in (0, 1)$, with probability at least $\delta$ that the point set $P$ with $N_{\text{tot}} = N\times M $ points, satisfies
$$
D_{N_{\text{tot}}}^*(P) \le 1.83 \times \frac{s\log(N+1) + \log 2 - \log(1-\delta) }{M}.
$$
This completes the proof.
\end{proof}

\subsection{Proof of Theorem \ref{thm:main_result2}}
The proof follows a similar structure to that of Theorem 1.  Consider $P_f=\bigcup_{r=1}^{M}\left(P_N(r) \oplus \bm{\Delta}_r\right)$, where $\bm{\Delta}_1, \dots, \bm{\Delta}_M \in [0,1)^s$ are chosen i.i.d. uniformly distributed. Similarly, we define
\begin{equation*}
    \mathcal{Y}_r(\vecb) := \mathrm{disc}(P_N(r) \oplus \bm{\Delta}_r, J(\vecb)),
\end{equation*}
where $\vecb \in \bar \Gamma^s$.
From Lemma \ref{lem:cont_ex_delta}, we have $\mathbb{E}_{\bm{\Delta}_r}[\mathcal{Y}_r(\vecb)] = 0$. Let $\mathcal{S}_{M}(\vecb) = \sum\limits_{r=1}^{M}\mathcal{Y}_r(\vecb)$, we have $\mathbb{E}[\mathcal{S}_{M}(\vecb)] = 0$. Then we consider $\text{Var}[\mathcal{S}_{M}(\vecb)] = \sum\limits_{r=1}^{M}\text{Var}[\mathcal{Y}_r(\vecb)] = \sum\limits_{r=1}^{M}\mathbb{E}_{\bm{\Delta}_r}[\mathcal{Y}_r^2(\vecb)]$.

\begin{lemma} \label{lem:varYd}
Let $\{\vecDelta_r\}_{r=1}^M$ be chosen uniformly at random from $[0,1)^s$. For any axis-parallel box $J(\vecb)$ with $\vecb \in \bar\Gamma^s$,
\begin{equation} \label{eq:variance_bound}
\sum_{r=1}^{M}\mathbb{E}_{\bm{\Delta}_r}[\mathrm{disc}^2(P_N(r) \oplus \bm{\Delta}_r, J(\vecb))] < s(1+1/(3N)). \nonumber
\end{equation}
\end{lemma}

\begin{proof}
Following the proof of Lemma \ref{lem:cont_var}, we consider $\emptyset \neq \bm{u} \subseteq \{1,2,\ldots,s\}$ and $\bm{b} = (\bm{b}_{\bm{u}}, \mathbf{1}_{\bar{\bm{u}}})$. Then, we have
\begin{align*}
    F_{r, (\bm{b}_{\bm{u}}, \mathbf{1}_{\bar{\bm{u}}})}(\bm{\Delta}) &:= \mathrm{disc}(P_N(r) \oplus \vecDelta, J(\vecb)) = \sum_{\bm{k}_{\bm{u}} \neq \boldsymbol{0}_{\vecu}} \left[c_{\bm{k}_{\bm{u}}}(\bm{b}_{\bm{u}}) S_N(r, \bm{k}_{\bm{u}})\right]e^{2\pi i \bm{k}_{\bm{u}} \cdot \boldsymbol{\Delta}_{\bm{u}}}.
\end{align*}
By continuous Parseval's identity and Lemma \ref{lem:sum_sn}, it follows that
\begin{equation*}
\begin{aligned}
\sum_{r=1}^{M}\mathbb{E}_{\vecDelta_r} [|F_{r, (\bm{b}_{\bm{u}}, \mathbf{1}_{\bar{\bm{u}}})}(\vecDelta_r)|^2] &= \sum_{r=1}^{M}\left[\sum_{\bm{k}_{\bm{u}} \neq \boldsymbol{0}_{\vecu}} |c_{\bm{k}_{\bm{u}}}(\bm{b}_{\bm{u}}) S_N(r, \bm{k}_{\bm{u}})|^2\right]\\
&= \sum_{\bm{k}_{\bm{u}} \neq \boldsymbol{0}_{\vecu}} |c_{\veck}(\vecb)|^2 \sum_{r=1}^{M} |S_N(r, \veck)|^2\\
& = \sum_{\bm{k}_{\bm{u}} \in L_1} |c_{\veck}(\vecb)|^2 \sum_{r=1}^{M} |S_N(r, \veck)|^2 + \sum_{\bm{k}_{\bm{u}} \in L_2} |c_{\veck}(\vecb)|^2 \sum_{r=1}^{M} |S_N(r, \veck)|^2\\
& \le (s-1)\sum_{\bm{k}_{\bm{u}} \in L_1} |c_{\veck}(\vecb)|^2+M\sum_{\bm{k}_{\bm{u}} \in L_2} |c_{\veck}(\vecb)|^2\\
& \le (s-1)\lambda(J(\vecb_{\bm{u}}))(1 - \lambda(J(\vecb_{\bm{u}})))+\frac{Ms}{3N^2} < s(1+1/(3N)),
\end{aligned}
\end{equation*}
where $L_{1} = \{\veck_{\bm{u}} \in \setZ^{|\bm{u}|} \setminus \{\boldsymbol{0}_{\bm{u}}\} : \veck_{\bm{u}} \not\equiv \boldsymbol{0} \pmod N\}$ and $L_{2} = (N\setZ)^{|\bm{u}|} \setminus \{\boldsymbol{0}_{\bm{u}}\}$.
This completes the proof.
\end{proof}

In summary, we have the following properties. For fixed $\vecb \in \bar\Gamma^s$, $\{\mathcal{Y}_r(\vecb)\}_{r=1}^{M}$ are i.i.d. and satisfy 
$$
|\mathcal{Y}_r(\vecb)| \le 1, \quad \mathbb{E}_{\bm{\Delta}_r}[\mathcal{Y}_r(\vecb)] = 0, \quad \text{Var}[\mathcal{S}_{M}(\vecb)] = \sum_{r=1}^{M}\mathbb{E}_{\bm{\Delta}_r}[\mathcal{Y}_r^2(\vecb)] < s(1+1/(3N)). 
$$
Thus the random variables $\mathcal{Y}_r(\vecb)$ satisfy the similar properties as $Y_r(\vecb)$ given in \eqref{eq:Yr}. Therefore, the results from the proof of Theorem \ref{thm:main_result1} apply accordingly.

\subsection{Proof of Theorem \ref{thm:main_result3}}
We now turn to the case of discrete shifts with random generators.
Consider $Q=\bigcup_{r=1}^{M}\left(P_N(z_r) \oplus \bm{\delta}_r\right)$, where $z_1, \dots, z_M \in \{1,2,\ldots,M\}$ and $\bm{\delta}_1, \dots, \bm{\delta}_M \in \Gamma^s$ are chosen i.i.d. uniformly distributed, respectively. For $\vecb \in \bar \Gamma^s$, we define
\begin{equation*}
    X_r(\vecb) := \mathrm{disc}(P_N(z_r) \oplus \bm{\delta}_r, J(\vecb)),
\end{equation*}
From Lemma \ref{lem:disc_ex_delta}, we have $\mathbb{E}_{\bm{\delta}_r}[X_r(\vecb)] = 0$. Let $T_{M}(\vecb) = \sum\limits_{r=1}^{M}X_r(\vecb)$, it follows that $\mathbb{E}[T_{M}(\vecb)]  = 0$. Then we consider $\text{Var}[T_{M}(\vecb)] = \sum\limits_{r=1}^{M}\text{Var}[X_r(\vecb)] = \sum\limits_{r=1}^{M}\mathbb{E}_{\bm{\delta}_r}[X_r^2(\vecb)]$.

\begin{lemma} \label{lem:varX}
Let $\vecdelta \in \Gamma^s$ and $z \in \{1,2,\ldots,M\}$ be chosen uniformly at random. For any axis-parallel box $J(\vecb)$ with $\vecb \in \bar\Gamma^s$,
\begin{equation} \label{eq:variance_bound}
\mathbb{E}_{z,\bm{\delta}}[\mathrm{disc}^2(P_N(z) \oplus \bm{\delta}, J(\vecb))] < \frac{s}{M}. \nonumber
\end{equation}
\end{lemma}

\begin{proof}
By the definition of local discrepancy \eqref{eq:local}, we have
\begin{align*}
    F_{z, \bm{b}}(\bm{\delta}) &:= \mathrm{disc}(P_N(z) \oplus \bm{\delta}, J(\vecb)) = \frac{1}{N} \sum_{n=1}^N \mathbf{1}_{J(\bm{b})}(\{ \mathbf{x}_{n} + \bm{\delta} \}) - \text{Vol}(J(\bm{b})) \\
    &= \frac{1}{N} \sum_{n=1}^N \left( \sum_{\bm{k}
    \in \setD_N^{s} \setminus \{\boldsymbol{0}\}} C_{\bm{k}_{\bm{u}}}(\bm{b}) e^{2\pi i \, \bm{k}\cdot (\mathbf{x}_{n} + \bm{\delta})} \right)\\
    &= \sum_{\bm{k}
    \in \setD_N^{s} \setminus \{\boldsymbol{0}\}}\left[C_{\bm{k}}(\bm{b}) S_N(z, \bm{k})\right]e^{2\pi i \bm{k}\cdot \boldsymbol{\delta}},
\end{align*}
where 
$
S_N(z, \bm{k}) := \frac{1}{N} \sum_{n=0}^{N-1} e^{2\pi i \bm{k}\cdot \boldsymbol{x}_{n}}$.
Since $\bm{\delta} \in \Gamma^s$, $F_{z, \bm{b}}(\bm{\delta})$ is a finite sum of bounded indicator functions in $L^2(\Gamma^s)$. Consequently, by discrete Parseval's identity and Lemma \ref{lem:ex_sn}, it follows that
\begin{equation*}
\begin{aligned}
    \mathbb{E}_{z,\vecdelta} [|F_{z, \bm{b}}(\vecdelta)|^2] &= \mathbb{E}_{z}\left[\sum_{\bm{k}
    \in \setD_N^{s} \setminus \{\boldsymbol{0}\}} |C_{\veck}(\vecb) S_N(z, \veck)|^2\right] = \sum_{\bm{k}
    \in \setD_N^{s} \setminus \{\boldsymbol{0}\}} |C_{\veck}(\vecb)|^2 \mathbb{E}_{z}|S_N(z, \veck)|^2\\
    & \le \frac{s-1}{M} \sum_{\bm{k}
    \in \setD_N^{s} \setminus \{\boldsymbol{0}\}} |C_{\veck}(\vecb)|^2 = \frac{s-1}{M}\lambda(J(\vecb))(1 - \lambda(J(\vecb))) < \frac{s}{M},
\end{aligned}
\end{equation*}
where the last inequality follows by Lemma \ref{lem:disc_square}.
This completes the proof.
\end{proof}

\begin{proof}[Proof of Theorem \ref{thm:main_result3}]
In summary, we have the following properties. For fixed $\vecb \in \bar\Gamma^s$, $\{X_r(\vecb)\}_{r=1}^{M}$ are i.i.d. and satisfy 
\begin{equation}\label{eq:Xr}
    |X_r(\vecb)| \le 1, \quad \mathbb{E}_{\bm{\delta}_r}[X_r(\vecb)] = 0, \quad \text{Var}[T_{M}(\vecb)] = \sum_{r=1}^{M}\mathbb{E}_{z_r,\bm{\delta}_r}[X_r^2(\vecb)] < s. 
\end{equation}
Thus the random variables $X_r(\vecb)$ satisfy similar properties as $Y_r(\vecb)$ given in \eqref{eq:Yr}. Thus the results from the proof of Theorem \ref{thm:main_result1} apply accordingly. In particular, the
bound \eqref{eq:MYr} applies also to $\frac{1}{M}\sum_{r=1}^{M}X_r(\vecb)$: with probability at least $\delta \in (0, 1)$ for all $\vecb$, it is true that
\begin{equation}\label{eq:XMr}
    \begin{aligned}
\left| \frac{1}{M} \sum_{r=1}^{M} X_r(\boldsymbol{b}) \right|
&\le \underbrace{\frac{1}{3} \left( 1 + \sqrt{1 + \frac{18}{\log 3}} \right)}_{=1.7231\ldots} \times \frac{s\log(N+1) + \log 2 - \log(1-\delta) }{M}.
\end{aligned}
\end{equation}
By \eqref{eq:discre} and \eqref{eq:XMr}, we obtain that for any $\delta \in (0, 1)$, with probability at least $\delta$ that the point set $Q$ with $N_{\text{tot}} = N\times M$ points, satisfies
$$
D_{N_{\text{tot}}}^*(Q) \le 1.73 \times \frac{s\log(N+1) + \log 2 - \log(1-\delta) }{M}.
$$
This completes the proof.
\end{proof}

\subsection{Proof of Theorem \ref{thm:main_result4}}
For the point set $Q_f=\bigcup_{r=1}^{M}\left(P_N(r) \oplus \bm{\delta}_r\right)$, where $\bm{\delta}_1, \dots, \bm{\delta}_M \in \Gamma^s$ are chosen i.i.d. uniformly distributed, similarly, we define
\begin{equation*}
    \mathcal{X}_r(\vecb) := \mathrm{disc}(P_N(r) \oplus \bm{\delta}_r, J(\vecb)),
\end{equation*}
where $\vecb \in \bar \Gamma^s$.
From Lemma \ref{lem:disc_ex_delta}, we have $\mathbb{E}_{\bm{\delta}_r}[\mathcal{X}_r(\vecb)] = 0$. Let $\mathcal{T}_{M}(\vecb) = \sum\limits_{r=1}^{M}\mathcal{X}_r(\vecb)$, we have $\mathbb{E}[\mathcal{T}_{M}(\vecb)] = 0$. Then we consider $\text{Var}[\mathcal{T}_{M}(\vecb)] = \sum\limits_{r=1}^{M}\text{Var}[\mathcal{X}_r(\vecb)] = \sum\limits_{r=1}^{M}\mathbb{E}_{\bm{\delta}_r}[\mathcal{X}_r^2(\vecb)]$.

\begin{lemma} \label{lem:varYd}
Let $\{\vecdelta_r\}_{r=1}^M$ be chosen uniformly at random from $\Gamma^s$. For any axis-parallel box $J(\vecb)$ with $\vecb \in \bar\Gamma^s$,
\begin{equation} \label{eq:variance_bound}
\sum_{r=1}^{M}\mathbb{E}_{\bm{\delta}_r}[\mathrm{disc}^2(P_N(r) \oplus \bm{\delta}_r, J(\vecb))] < s. \nonumber
\end{equation}
\end{lemma}

\begin{proof}
Following the proof of Lemma \ref{lem:varX}, we have
\begin{align*}
    F_{r, \bm{b}}(\bm{\delta}_r) &:= \mathrm{disc}(P_N(r) \oplus \bm{\delta}_r, J(\vecb)) = \sum_{\bm{k}
    \in \setD_N^{s} \setminus \{\boldsymbol{0}\}}\left[C_{\bm{k}}(\bm{b}) S_N(r, \bm{k})\right]e^{2\pi i \bm{k}\cdot \boldsymbol{\delta}_{r}}.
\end{align*}
Consequently, by the discrete Parseval's identity and Lemma \ref{lem:sum_sn}, it follows that
\begin{equation*}
\begin{aligned}
\sum_{r=1}^{M}\mathbb{E}_{\vecdelta_r} [|F_{r, \bm{b}}(\vecdelta_r)|^2] &= \sum_{r=1}^{M}\sum_{\bm{k}
    \in \setD_N^{s} \setminus \{\boldsymbol{0}\}} |C_{\veck}(\vecb) S_N(r, \veck)|^2 = \sum_{\bm{k}
    \in \setD_N^{s} \setminus \{\boldsymbol{0}\}} |C_{\veck}(\vecb)|^2 \sum_{r=1}^{M}\left[ |S_N(r, \veck)|^2\right]\\
    & \le (s-1) \sum_{\bm{k}
    \in \setD_N^{s} \setminus \{\boldsymbol{0}\}} |C_{\veck}(\vecb)|^2 = (s-1)\lambda(J(\vecb))(1 - \lambda(J(\vecb))) < s,
\end{aligned}
\end{equation*}
where the last inequality follows by Lemma \ref{lem:disc_square}.
This completes the proof.
\end{proof}

In summary, we have the following properties. For fixed $\vecb \in \bar\Gamma^s$, $\{\mathcal{X}_r(\vecb)\}_{r=1}^{M}$ are i.i.d. and satisfy 
$$
|\mathcal{X}_r(\vecb)| \le 1, \quad \mathbb{E}_{\bm{\delta}_r}[\mathcal{X}_r(\vecb)] = 0, \quad \text{Var}[\mathcal{T}_{M}(\vecb)] = \sum_{r=1}^{M}\mathbb{E}_{\bm{\delta}_r}[\mathcal{X}_r^2(\vecb)] < s. 
$$
Thus the random variables $\mathcal{X}_r(\vecb)$ satisfy the similar properties as $X_r(\vecb)$ given in \eqref{eq:Xr}. Thus the results from the proof of Theorem \ref{thm:main_result3} apply accordingly. 





\bibliographystyle{mybst} 
\bibliography{references}



\end{document}